\documentclass[11pt]{article}
\usepackage[a4paper]{geometry}
\usepackage{amsfonts,amsmath,amssymb,amsthm,graphicx,hyperref,xcolor}

\title{Accuracy and stability of the hyperbolic model\\time integration scheme revisited}
\author{Mikhail A. Botchev}

\begin{document}
\maketitle

\begin{abstract}

The hyperbolic model (HM) time integration scheme tackles parabolic problems
by adding a small artificial second order time derivative term.  
Described by Samarskii in his 1971 book, the scheme reappeared  
as the generalized Du~Fort--Frankel scheme
in a 1976 paper by Gottlieb and Gustafsson.
In this note we revisit accuracy and stability properties of the scheme.
In particular, we show that the stability condition, formulated
by Samarskii based on operator inequalities, coincides
with the requirement that the eigenvalues of the amplification matrix
(the stability function operator) are smaller than one
in absolute value.  However,
under this condition, the norm of this matrix may exceed one and 
this, as recently pointed out by Corem and Ditkowski (2012), may 
corrupt convergence of the scheme.
Hence, we also discuss whether this eventual stability lack can be detected and mitigated in practice.




\end{abstract}

\maketitle

\newcommand{\Cc}{\mathbb{C}}
\newcommand{\eps}{\epsilon}
\newcommand{\epst}{\tilde{\varepsilon}}
\newcommand{\et}{\tilde{e}}
\newcommand{\geqs}{\geqslant}
\newcommand{\leqs}{\leqslant}
\newcommand{\Rr}{\mathbb{R}}
\newcommand{\vareps}{\varepsilon}
\newcommand{\yh}{\hat{y}}
\newcommand{\yt}{\tilde{y}}

\newtheorem{proposition}{Proposition}
\newtheorem{remark}{Remark}

\section{Introduction}
We aim at solving an initial value problem (IVP) 
\begin{equation}
\label{IVP}
y'(t) = - A y(t) + f(t), \qquad y(0) =y^0,
\qquad 0\leqs t\leqs T,  
\end{equation}
where a symmetric positive semidefinite matrix $A=A^T\in\Rr^{N\times N}$,
a vector function $f:[0,T]\rightarrow \Rr^N$ and a vector $y^0$ are given.
Our interest concerns problems~\eqref{IVP} which arise from a space discretization
of parabolic PDE, so that $A$ is a discretized partial derivative operator
and usually has a stiff spectrum, i.e., we can expect that the largest
eigenvalue $\lambda_{\max}$ of~$A$ is of order $h^{-2}$, with $h$ being
the space grid size, and its smallest eigenvalue $\lambda_{\min}\geqs 0$
can be (arbitrarily close to) zero. 

Our study subject in this paper is the following time integration
scheme for solving~\eqref{IVP}, 
usually referred to as the hyperbolic model (HM) 
scheme~\cite{Samarskii_book1971,ZlotnikChetverushkin2008,CheOlkGas23}:
\begin{equation}
\label{HM}
\frac{y^{n+1}-y^{n-1}}{2\tau} + 
\vareps\frac{y^{n+1} -2y^n + y^{n-1}}{\tau^2} = -Ay^n + f^n,
\qquad n=1,2,\dots, 
\end{equation}
where $\tau>0$ is the time step size, $\vareps>0$ is a small parameter,
$f^n=f(t_n)$, $t_n=\tau n$, and $y^n$ is the numerical solution vector
at time $t_n$.
The solution $y^1$ at the first time level can be obtained by a step
of any suitable time integration scheme, e.g., by the explicit Euler method.
The scheme~\eqref{HM}
is attractive due its simplicity, explicit structure and a relaxed time
step restriction as compared to conventional explicit schemes.  
However, the scheme performance heavily depends on a proper choice of the
small parameter in the artificially added hyperbolic term.  
Although there are some, mostly heuristic, approaches for choosing
the small parameter (e.g., based on 
``the characteristic rate of the process''~\cite{ChetverushkinOlkhovskayaGasilov23a}
or heat front speed estimates~\cite{CheOlkGas23}),
a proper choice of~$\vareps$ in practice remains an important issue.

As far as we know, the HM scheme~\eqref{HM} was for the first time considered
in~1971 by Samarskii~\cite[Chapter~VI, \S~2]{Samarskii_book1971}.
There in formula~(80), for $\kappa>0$, a finite difference scheme
is given
\begin{equation}
\label{HM1}
\frac{y^{n+1}-y^{n-1}}{2\tau} + 
\kappa\tau^2\frac{y^{n+1} -2y^n + y^{n-1}}{\tau^2} = -Ay^n + f^n,
\qquad n=1,2,\dots, 
\end{equation}
which is the HM scheme~\eqref{HM} written for $\vareps=\kappa\tau^2$.  
Samarskii also
provided~\cite[Chapter~VI, \S~2, formula~(81)]{Samarskii_book1971}
a sufficient stability condition for the HM scheme:
\begin{equation}
\label{stab1}
\kappa>\frac14\|A\|
\quad\stackrel{\;\vareps=\kappa\tau^2}{\Longleftrightarrow}\quad
\tau< \sqrt{4\vareps\|A\|^{-1}},
\end{equation}
where $\|A\|$ is an operator norm (i.e., a matrix norm induced
by a vector norm).
As is further pointed out in~\cite[Chapter~VI, \S~2]{Samarskii_book1971},
if IVP~\eqref{IVP} is obtained by the standard three-point finite difference
space discretization of a 1D heat equation, i.e.,
if $(-Ay)_i=(y_{i+1}-2y_i+y_{i+1})/h^2$ at inner grid points~$i$, then 
the well-known scheme of Du~Fort and Frankel~\cite{DuFortFrankel1953}
is a particular case of the HM scheme~\eqref{HM} for the choice
$\vareps = \tau^2/h^2$.
Taking into account that $\lambda_{\max}=\|A\|_2< 4/h^2$,
Samarskii remarks that his stability condition~\eqref{stab1} is
fulfilled for the Du~Fort--Frankel scheme and, thus,
the scheme is unconditionally (for any $\tau>0$ and $h>0$) stable.
In~1973,
a stability condition similar to~\eqref{stab1} was derived
for an implicit-explicit generalization of~\eqref{HM},
see~\cite[page~238]{SamarskiiGulin_book1973}; 
when applied to scheme~\eqref{HM}, the stability condition 
formula appearing there is equivalent to~\eqref{stab1}.
  
Note that a link from the original Du~Fort--Frankel 
scheme to a hyperbolical perturbation of the 1D heat equation problem, 
is made already in the 1967 classical book~\cite[Chapter~7.5]{RichtMorton}.
As noted there, the Du~Fort--Frankel scheme is consistent only
if $\tau/h\rightarrow 0$ for $\tau\rightarrow 0$.
In 1976, the HM scheme~\eqref{HM} 
reappeared in~\cite{GottliebGustafsson1976}, where the
motivation was to generalize the original Du~Fort--Frankel scheme
to an arbitrary space dimension, space discretization, and number of
equations.
The authors called their method the generalized Du Fort--Frankel
scheme~\cite[formula~(2.2)]{GottliebGustafsson1976}
and it is the HM scheme applied to~\eqref{IVP} 
for $\vareps=\gamma\sigma\tau^2/h^2$,
where $\gamma>0$ is a small parameter and $\sigma>0$ is the heat
conduction coefficient.  Furthermore, the authors~\cite{GottliebGustafsson1976}
state that the scheme is consistent provided $\tau/h\rightarrow 0$ and 
derive a stability condition which coincides with
the sufficient stability condition~\eqref{stab1}, given by Samarskii in 1971,
and is claimed to guarantee that the scheme is stable
unconditionally, i.e., for any $\tau>0$ and $h>0$.

As mentioned in~\cite[Chapter~IV.3.1]{HundsdorferVerwer:book},
the original Du Fort--Frankel scheme can be seen as a particular case
of the odd-even Hopscotch method, i.e., a combination of the 
red-black (checkerboard) splitting in space with the ADI 
Peaceman--Rachford time splitting.  Taking into account
that the Du Fort--Frankel scheme is consistent only for $\tau/h\rightarrow 0$,
the authors note that its stability properties is ``of very
limited value''.
They comment that, although the scheme is still occasionally used
to solve problems with very small diffusion 
coefficients to a low accuracy, ``in general it is not a scheme to
be recommended''~\cite[Chapter~IV.3.1]{HundsdorferVerwer:book}.
Recently, Corem and Ditkowski~\cite{CoremDitkowski2012} studied 
accuracy and stability properties of the generalized Du Fort--Frankel
scheme~\cite{GottliebGustafsson1976}.  
They show that, although the scheme has nice stability
properties in the von Neumann sense, the powers of the scheme 
amplification matrix may exhibit a significant growth in norm.
Moreover, this can, in some special cases, lead 
to convergence problems~\cite{CoremDitkowski2012}.
Convergence properties of the HM scheme for nonsmooth
initial data are studied in~\cite{IlyinRykov2016,Moiseev_ea2018}.

The hyperbolic correction, also known as the hyperbolization, 
i.e., modifying PDEs by adding a small second order time derivative term, 
has been widely used for deriving PDE based mathematical models as well as
for treating PDE systems numerically.
In particuar, this approach has been employed in fluid dynamics 
in the context of the so-called quasi-gasdynamic (QGD) systems
for modeling unsteady turbulent flows, combustion processes 
and solving problems in aeroelasticity and 
aeroacoustics~\cite{Chetverushkin2004book,ZlotnikChetverushkin2008}.
Explicit structure of the HM scheme makes it attractive as a
computational tool for modern high-performance computing 
systems~\cite{ChetverushkinGulin2012}. 
Considering a model heat equation with a heat conduction
coefficient~$K$ and substituting $\|A\|=4K/h^2$ into the stability 
condition~\eqref{stab1}, 
the authors~\cite{ChetverushkinGulin2012} obtain a stability estimate
\begin{equation}
\label{stab1a}
\tau < \sqrt{4\vareps\|A\|^{-1}} = \sqrt{\frac{4\vareps}{4K/h^2}} =
h\sqrt{\frac{\vareps}{K}}
\end{equation}
and propose to choose $\vareps=Kh$, so that
\begin{equation}
\label{restr1}
\tau < h^{3/2}.
\end{equation}
Clearly, this is a much weaker restriction than the usual $O(h^2)$ 
stability condition of the conventional explicit schemes.
In~\cite{ChetverushkinOlkhovskaya2020}, the HM scheme is applied
to model radiative heat conduction processes on high-performance computing 
systems.  Furthermore, the HM scheme is demonstrated to successfully
solve nonlinear heat conduction equations~\cite{CheOlkGas23}
and radiation diffusion equation~\cite{ChetverushkinOlkhovskayaGasilov23a}.
The authors in~\cite{ChetverushkinOlkhovskaya2020,CheOlkGas23,ChetverushkinOlkhovskayaGasilov23a} use the stability condition~\eqref{stab1a} and
a suitable choice $\vareps=O(h)$, with the proportionality
constant reciprocal to ``the characteristic rate of the process''. 
Hence, a restriction similar to~\eqref{restr1} is obtained.   
We briefly comment on this choice below at the end of Section~\ref{s:num}. 

The aim of this paper is reevaluate accuracy and stability properties
of the HM scheme, taking into account recent results 
of~\cite{CoremDitkowski2012}.
In particular, we present an accuracy analysis of the scheme, separating 
the scheme error into an error triggered by adding the small hyperbolic term
and the truncation error of the scheme with respect to the 
hyperbolically perturbed problem.  Furthermore, we carry out 
a stability analysis of the scheme, which shows, in particular,
that, for symmetric positive definite matrices~$A$, 
the Samarskii stability condition is necessary and sufficient
for the eigenvalues of the amplification matrix be smaller than one
in absolute value.
However, a growth of the amplification matrix powers can occur also
if the eigenvalues are smaller than one in absolute value.  Therefore,
we discuss whether this eventual stability lack can be detected and 
mitigated in practice.

Unless explicitly stated otherwise, throughout
this paper, $\|A\|$ denotes the spectral norm, i.e., the operator norm
$\|A\|=\max_{\|x\|=1}\|Ax\|$ with $\|x\|$ being the Euclidean vector norm
(the 2-norm).
Furthermore, in this paper we use a function
\begin{equation}
\label{phi}
\varphi(t) \equiv \frac{e^t - 1}{t}, \quad t\in\Cc,
\end{equation}
where it is assumed that $\varphi(0)=1$.
Next, by $\omega=\lambda_{\min}\geqs 0$ we denote the smallest eigenvalue
of~$A$.  
There exists a constant $C$ 
(see, e.g.,~\cite[relation~(2.6)]{HochbruckOstermann2010})
such that
\begin{equation}
\label{exp_est}
\|e^{-tA}\|_* \leqs C e^{-\omega t}, \quad t\geqs 0,
\end{equation}
where $e^{-tA}$ is the matrix exponential of $-tA$,
$\|\cdot\|_*$ is an operator norm and the constant
$C$ depends on $A$ and the matrix norm $\|\cdot\|_*$.
For the spectral matrix norm we have $C=1$.

This paper is organized as follows.
In Section~\ref{s:acc_stab} accuracy and stability properties
of the HM scheme are studied. There, in Section~2.1 
we derive
an error estimate for a perturbed problem with a small hyperbolical
term added.  Then, based on a standard truncation error analysis of the
HM scheme with respect to the hyperbolically perturbed problem, conclusions
are made on expected convergence order of the HM scheme with respect to
the original IVP~\eqref{IVP}.
Section~2.2 
is devoted to studying stability properties of the HM scheme,
where we first formulate a result on the eigenvalues of the scheme amplification
matrix~$S$ and then discuss its monotonicity properties and eventual growth
of the powers~$S^n$ in norm.  Section~\ref{s:num} describes numerical tests.
Conclusions are drawn in Section~\ref{s:concl}.

\section{Accuracy and stability of the HM scheme}
\label{s:acc_stab}

\subsection{Local error of the HM scheme}
\label{ss:acc}
In this section we derive an expression for the local error of the HM scheme.
Let $y(t)$ be the exact solution of IVP~\eqref{IVP} 
and let 
\begin{equation}
\label{le_def}
e^{n+1}\equiv y(t_{n+1})-y^{n+1}, \quad n=1,2,\dots,
\end{equation}
be the local error of the HM scheme.  
As usual for a local error analysis 
(see, e.g.,~\cite{ODE1_HNW,HundsdorferVerwer:book}),
we assume that $e^{n+1}$ is an error committed after a step of the HM scheme,
starting with the exact values $y^{n-1}=y(t_{n-1})$, $y^n=y(t_n)$.

Together with IVP~\eqref{IVP}, we also consider its hyperbolical
model approximation (the so-called HM approximation)
\begin{equation}
\label{IVP_HM}
\yt'(t) + \vareps\yt''(t) = - A \yt(t) + f(t), \qquad
\yt(0) = y^0,\; \yt'(0)=v^0,
\qquad 0\leqs t\leqs T,  
\end{equation}
for the unknown function $\yt(t)$, where $v^0$ is given, and 
the notation is the same as in~\eqref{IVP}.
Let $\et(t)$ be the error of the HM approximation, i.e., define
$\et(t)$ as 
\begin{equation}
\label{et}
\et(t) \equiv y(t) - \yt(t), \quad t\geqs 0,
\end{equation}
where $y(t)$ and $\yt(t)$ are the exact solutions of the IVPs~\eqref{IVP}
and~\eqref{IVP_HM}, respectively.  The following result provides
an upper estimate for the HM error $\et(t)$
(similar estimates for weak forms of~\eqref{IVP}
can be found in~\cite{RepinChetverushkin2013}).

\begin{proposition}
\label{P_et}
For the HM error $\et(t)$, defined in~\eqref{et}, holds
\begin{equation}
\label{et_est}
\|\et(t)\| \leqs C\vareps t\varphi(-\omega t) \max_{s\in[0,t]}\|\yt''(s)\|
=
\begin{cases}
\;C\vareps t\max_{s\in[0,t]}\|\yt''(s)\|,\quad &\text{for} \quad \omega=0,
\\
\;C\vareps \dfrac{1-e^{-\omega t}}{\omega}\max_{s\in[0,t]}\|\yt''(s)\|,
\quad &\text{for} \quad \omega>0,
\end{cases}
\end{equation}
where $\|\cdot\|$ is any vector norm, the function $\varphi(t)$ is defined
in~\eqref{phi}, $\omega\geqs 0$ is the smallest eigenvalue of~$A$, 
and the constant $C$ is defined in~\eqref{exp_est}. 
\end{proposition}

\begin{proof}
Subtracting the equation in IVP~\eqref{IVP_HM} from the
equation in IVP~\eqref{IVP} and taking into account initial conditions
for $y(t)$ and $\yt(t)$, we see that $\et(t)$ satisfies an IVP
$$
\et'(t) = -A \et(t) + \vareps\yt''(t), \quad \et(0)=0.
$$
Then by the variation-of-constants formula
(see, for instance, \cite[Section~I.2.3]{HundsdorferVerwer:book} or 
\cite{HochbruckOstermann2010}) and inequality~\eqref{exp_est},
we have, for $t\geqs 0$,
\begin{align*}
\et(t)     &= \int_0^t e^{-(t-s)A}\vareps\yt''(s)ds, 
\\
\|\et(t)\| &\leqs \vareps\int_0^t \|e^{-(t-s)A}\|\,\|\yt''(s)\|ds
\leqs \vareps\int_0^t \|e^{-(t-s)A}\|ds \,\max_{s\in[0,t]}\|\yt''(s)\|
\\
&\leqs \vareps C \int_0^t e^{-\omega(t-s)} ds \,\max_{s\in[0,t]}\|\yt''(s)\|
= \vareps C t\varphi(-\omega t) \max_{s\in[0,t]}\|\yt''(s)\|.
\end{align*}
\end{proof}

Let now the IVPs~\eqref{IVP} and~\eqref{IVP_HM} be solved
for $t\in[t_n,t_{n+1}]$ with exact initial values
$y(t_n)$ and $\yt(t_n)=y(t_n)$, $\yt'(t_n)=y'(t_n)$, respectively.    
Let, furthermore, $\et^{n+1}(\tau)$ be the HM error defined by~\eqref{et},
with $t\geqs t_n$,
for the solutions of these IVPs.
Applying the estimate~\eqref{et_est} to $\et^{n+1}(\tau)$, it is not
difficult to see that
\begin{equation}
\label{etn_est}
\|\et^{n+1}(\tau)\| \leqs C\vareps \tau\varphi(-\omega\tau) 
\max_{s\in[t_n,t_{n+1}]}\|\yt''(s)\|.
\end{equation}

At this point, it is convenient to introduce 
truncation errors of the central finite differences appearing 
in the HM scheme: 
\begin{equation}
\label{eps1_2}
\eps_1^n \equiv \yt'(t_n) - \frac{\yt^{n+1}-\yt^{n-1}}{2\tau},
\quad
\eps_2^n  \equiv \yt''(t_n) - \frac{\yt^{n+1} -2\yt^n + \yt^{n-1}}{\tau^2},
\quad  n=0,1,2,\dots \, .
\end{equation}
It is not difficult to check, see, 
e.g., \cite[Chapter~II, Section~1]{Samarskii_book}, that,
provided $\yt(t)$ is a sufficiently smooth function, 
for any $t\geqs 0$ and $\tau>0$ there exist points $\eta_{1,2}\in[t-\tau, t+\tau]$ 
for which
\begin{align*}
\frac{\yt(t+\tau) - \yt(t-\tau)}{2\tau} 
&= \yt'(t) + \frac{\tau^2}{6}\yt'''(\eta_1), 
\\
\frac{\yt(t+\tau) - 2\yt(t) + \yt(t-\tau)}{\tau^2} 
&= \yt''(t) + \frac{\tau^2}{12}\yt^{(4)}(\eta_2).
\end{align*}
Thus, for every $n=0,1,2,\dots$, there exist points 
$\eta_{1,2}^n\in[t_{n-1},t_{n+1}]$ such that
\begin{equation}
\label{eps1_2a}
\eps_1^n = - \frac{\tau^2}{6}\yt'''(\eta_1^n),
\quad
\eps_2^n = - \frac{\tau^2}{12}\yt^{(4)}(\eta_2^n).
\end{equation}

\begin{proposition}
Assume the HM time integration scheme~\eqref{HM} is applied to
solve IVP~\eqref{IVP} and the exact solution~$y(t)$ of IVP~\eqref{IVP} 
is a sufficiently smooth function.
Then for the local error $e^{n+1}$ of the HM scheme, defined by~\eqref{le_def},
holds, for $n=1,2,\dots$,
\begin{equation}
\label{le}
\begin{gathered}
e^{n+1} = C_1 \frac{2\tau^4}{\tau + 2\vareps} 
+ C_2 \frac{2\tau^4\vareps}{\tau + 2\vareps} 
+ \et^{n+1}(\tau),
\\
C_1 = \frac{1}{6}\yt'''(\eta_1^n), \quad
C_2 = \frac{1}{12}\yt^{(4)}(\eta_2^n),
\end{gathered}
\end{equation}
where $\eta_{1,2}^n$ are certain points in the closed interval 
$[t_{n-1},t_{n+1}]$ and $\et^{n+1}(\tau)$ is the HM error, defined
as discussed before relation~\eqref{etn_est}, and for which
holds
$$
\|\et^{n+1}(\tau)\| \leqs C\vareps \tau\varphi(-\omega\tau) 
\max_{s\in[t_n,t_{n+1}]}\|\yt''(s)\|.
$$ 
\end{proposition}

\begin{proof}
Note that
$$
e^{n+1} = \underbrace{y(t_{n+1}) - \yt(t_{n+1})}_{\et^{n+1}(\tau)} 
+  \yt(t_{n+1}) - y^{n+1}.
$$
The rest of the proof is straightforward: we write down the HM scheme 
for exact values 
$y^{n-1}=\yt(t_{n-1})$, $y^n=\yt(t_n)$ and apply 
relations~\eqref{eps1_2a}.
\end{proof}

From~\eqref{le} we see that the local error of the HM scheme 
with respect to the HM approximation~\eqref{IVP_HM} has order~$4$ 
for $\vareps$ kept constant.  Below we discuss a choice $\vareps=\tilde{C}\tau$,
for which the local error of the HM scheme has order~$3$.
Hence, we expect that the global (i.e., computed after many
steps at the final time) time error of the scheme
with respect to the HM approximation~\eqref{IVP_HM}
should converge either with order~$3$ (for constant $\vareps$),
or with order~$2$ (for $\vareps=\tilde{C}\tau$).

It is important to realize that the estimates~\eqref{et_est},\eqref{etn_est}
for the HM error hold for any time interval length $t$
(respectively, for any $\tau>0$), i.e., the HM scheme is first
order consistent with respect to the original problem~\eqref{IVP}
for constant $\vareps$ and second order consistent
for $\vareps=\tilde{C}\tau$.

\subsection{Stability of the HM scheme}
\label{ss:stab}
In this note stability is investigated only with respect to the initial
data, i.e., it is assumed that problem~\eqref{IVP} is homogeneous 
($f\equiv 0$).  Stability with respect to the source function~$f$ 
can then be derived by standard considerations, 
see, e.g.,~\cite{GodunRyab,SamarskiiGulin_book1973}.

Introducing vectors
$$
Y^n = 
\begin{bmatrix}
y^n \\ y^{n-1}
\end{bmatrix}\in\Rr^{2N}, \quad n=1,2,\dots,
$$
we can write the HM scheme~\eqref{HM} as
\begin{equation}
\label{HMop}
\begin{gathered}
Y^{n+1} = S Y^n, \quad n=1,2,\dots,
\\
S =
\begin{bmatrix}
\dfrac{4\epst}{1+2\epst}I - \dfrac{2\tau}{1+2\epst} A & 
\dfrac{1-2\epst}{1+2\epst}I
\\
I & 0
\end{bmatrix}\in\Rr^{2N\times 2N},
\quad \epst=\frac{\vareps}{\tau}>0,
\end{gathered}
\end{equation}
where $I$ is the identity $N\times N$ matrix. 

\begin{proposition}
\label{p:stab}
If the matrix $A$ in~\eqref{HM} is symmetric positive definite
then the HM scheme~\eqref{HM} is stable in the sense that the
eigenvalues of its amplification matrix $S$ in~\eqref{HMop}
are smaller than~$1$ in absolute value
if and only if 
\begin{equation}
\label{stab2}
\vareps > \tau^2\dfrac{\lambda_{\max}}4
\quad\Leftrightarrow\quad
\tau<\sqrt{\frac{4\vareps}{\lambda_{\max}}},
\end{equation}
where $\lambda_{\max}$ is the largest eigenvalue of~$A$.

If the matrix $A$ in~\eqref{HM} is symmetric positive semidefinite
then each zero eigenvalue $\lambda_j=0$ of $A$ leads to a pair
eigenvalues $\xi_{1,2}^{(j)}$ in the amplification matrix~$S$,
\begin{equation}
\label{zero_eig}
\xi_1^{(j)}=1,\quad 
\xi_2^{(j)}=-\frac{1-2\epst}{1+2\epst}, \quad 
\epst=\frac{\vareps}{\tau}>0,
\end{equation} 
and all the eigenvalues
$\xi_{1,2}^{(j)}$ of~$S$, for all~$j$, are simple.
The other eigenvalues of~$S$ are smaller in absolute
value than~$1$ if and only if~\eqref{stab2} holds.    
\end{proposition}

\begin{remark}
Note that, since $\|A\|_2=\lambda_{\max}$ for symmetric positive semidefinite 
$A$, condition~\eqref{stab2} coincides with the Samarskii
sufficient stability condition~\eqref{stab1} which is 
obtained for a special vector norm by operator inequality
techniques.
The inequality~\eqref{stab2} can also be derived within assumptions 
of the Fourier-based (von Neumann stability) 
techniques~\cite{GottliebGustafsson1976,CoremDitkowski2012}.  
\end{remark}

\begin{proof}
Since the matrix $A$ in~\eqref{IVP} is real and symmetric, there exists
$Q\in\Rr^{N\times N}$ with orthonormal columns, which are the eigenvectors
of $A$, such that
\begin{equation}
\label{diag}
Q^TAQ = \Lambda = \mathop{\mathrm{Diag}} 
\begin{pmatrix}
\lambda_1 , \dots, \lambda_N 
\end{pmatrix}.
\end{equation}
Here $\Lambda$ is a diagonal $N\times N$ matrix whose entries
$\lambda_1$, \dots, $\lambda_N$ are the nonnegative eigenvalues of $A$. 
It is easy to check that
$$
\widehat{Q}^T S \widehat{Q} = 
\begin{bmatrix}
\dfrac{4\epst}{1+2\epst}I - \dfrac{2\tau}{1+2\epst} \Lambda & 
\dfrac{1-2\epst}{1+2\epst}I
\\
I & 0
\end{bmatrix},
\quad\text{for}\quad
\widehat{Q}= 
\begin{bmatrix}
Q & 0 \\
0 & Q
\end{bmatrix},
$$ 
and a permutation matrix $\widehat{P}\in\Rr^{2N\times 2N}$ exists
such that the matrix $\widehat{P}^T\widehat{Q}^TS\widehat{Q}\widehat{P}$
is a block diagonal with $2\times 2$ diagonal blocks
\begin{equation}
\label{Sj}
S_j=\begin{bmatrix}
\dfrac{4\epst - 2\mu_j}{1+2\epst} & 
\dfrac{1-2\epst}{1+2\epst}
\\
1 & 0
\end{bmatrix}, 
\quad \mu_j = \tau\lambda_j,
\quad j=1,\dots,N.
\end{equation}
Thus, the matrix 
$\widehat{P}^T\widehat{Q}^TS\widehat{Q}\widehat{P}$ is similar to~$S$
and the spectrum of~$S$ is formed by the union of the eigenvalues 
of the $2\times 2$ matrices $S_j$, for $j=1,\dots,N$.
In the rest of the proof, whenever convenient, we omit the subindex
$j$ in $\lambda_j$ and $S_j$.
The characteristic polynomial for each of the matrices~$S_j$
and its roots read, respectively, 
\begin{gather}
\notag
(1+ 2\epst)\xi^2 + (2\mu - 4\epst)\xi + (2\epst - 1) = 0,
\quad \text{with}\;\mu=\tau\lambda,
\\
\label{xi12}
\xi_{1,2} = 
\frac{2\epst - \mu \pm\sqrt{ (2\epst - \mu)^2 + 1 - 4\epst^2}}
{1+ 2\epst}.
\end{gather}
It is easy to see check that for $\lambda=0$ the eigenvalues
$\xi_{1,2}$ are given by~\eqref{zero_eig}.  
Now consider the eigenvalues $\xi_{1,2}$ corresponding to
the positive eigenvalues $\lambda>0$, so that $\mu>0$.
We have $|\xi_1\xi_2|= |(2\epst - 1)/(1 + 2\epst)|<1$ for $\epst>0$.
Hence, if $\xi_{1,2}$ are not real then $\xi_1=\bar{\xi}_2$ 
and $|\xi_1|=|\xi_2|<1$.
Therefore $|\xi_{1,2}|\geqs 1$ is possible only if $\xi_{1,2}\in\Rr$.
Assume $\xi_{1,2}\in\Rr$ and let $\xi_1\geqs \xi_2$,
so that $\xi_1$ corresponds to the plus sign
in the formula for $\xi_{1,2}$ above.
Thus, we have to check when $-1<\xi_2\leqs\xi_1<1$.
Imposing condition $\xi_1<1$, we get
\begin{gather*}
2\epst - \mu + \sqrt{ (2\epst - \mu)^2 + 1 - 4\epst^2} < 1+ 2\epst
\quad\Leftrightarrow\quad
2\mu(1+2\epst)>0,
\end{gather*}
which holds for all $\mu>0$ and $\epst>0$.
Imposing the other stability condition, $-1<\xi_2$, we have
\begin{gather*}
1 + 2\epst + 2\epst - \mu > \sqrt{ (2\epst - \mu)^2 + 1 - 4\epst^2} 
\quad\Leftrightarrow
\\
2\epst + 4\epst^2 + (1+2\epst)(2\epst + \mu) > 0
\quad\Leftrightarrow\quad
4\epst>\mu
\quad\Leftrightarrow\quad
4\frac{\vareps}{\tau}>\tau\lambda.
\end{gather*}
Requiring that the last inequality holds for all $\lambda=\lambda_j$,
we arrive at condition~\eqref{stab2}. 
\end{proof}

We note that, bringing the amplification matrix~$S$ to a block-diagonal form 
with $2\times 2$ blocks~$S_j$ (see relations~\eqref{HMop},\eqref{Sj}),
we follow~\cite{CoremDitkowski2012}, where this technique is
employed for a detailed stability analysis of the Du Fort--Frankel schemes.

\begin{figure}
\includegraphics[width=0.48\linewidth]{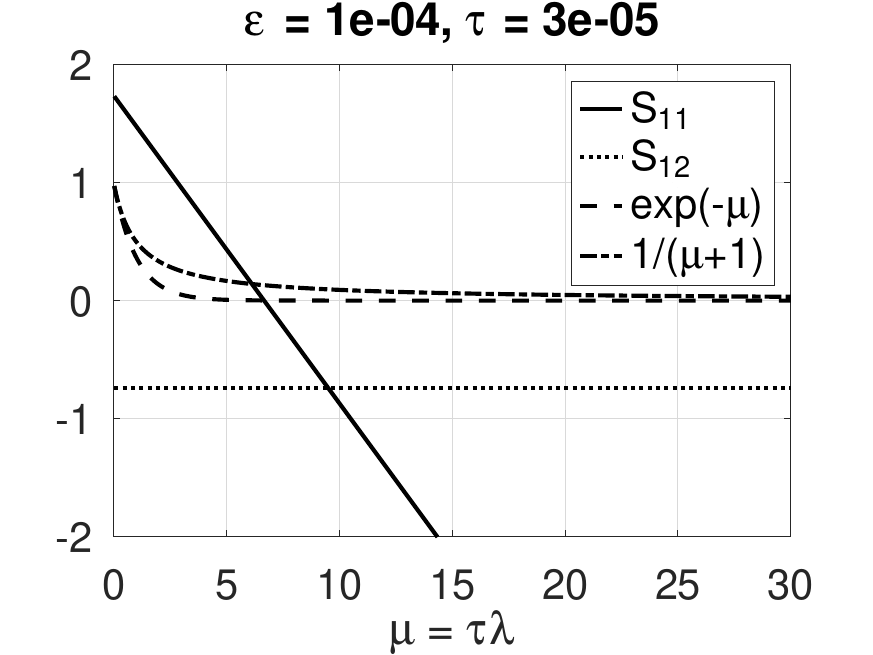}%
\includegraphics[width=0.48\linewidth]{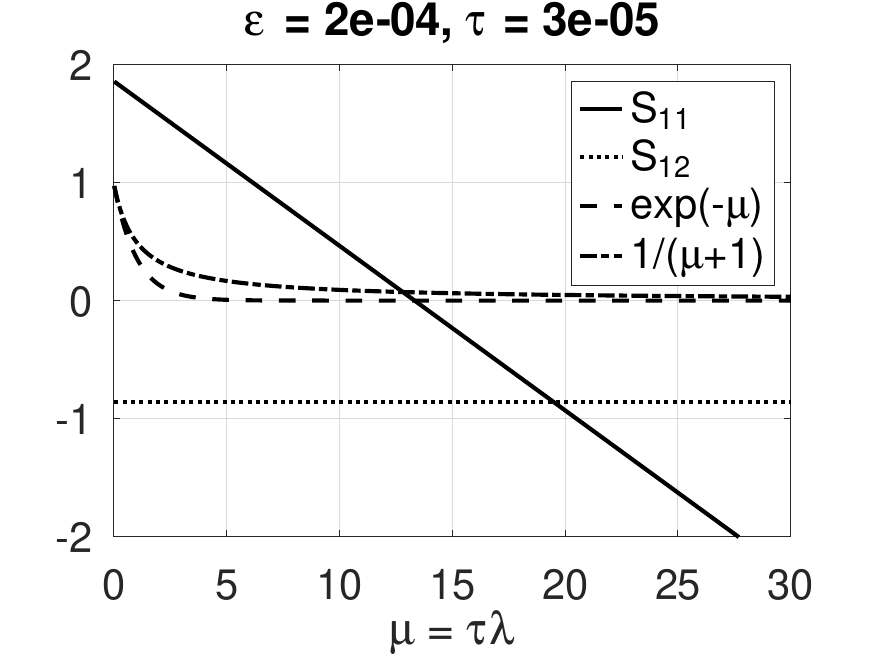}\\[1ex]
\includegraphics[width=0.48\linewidth]{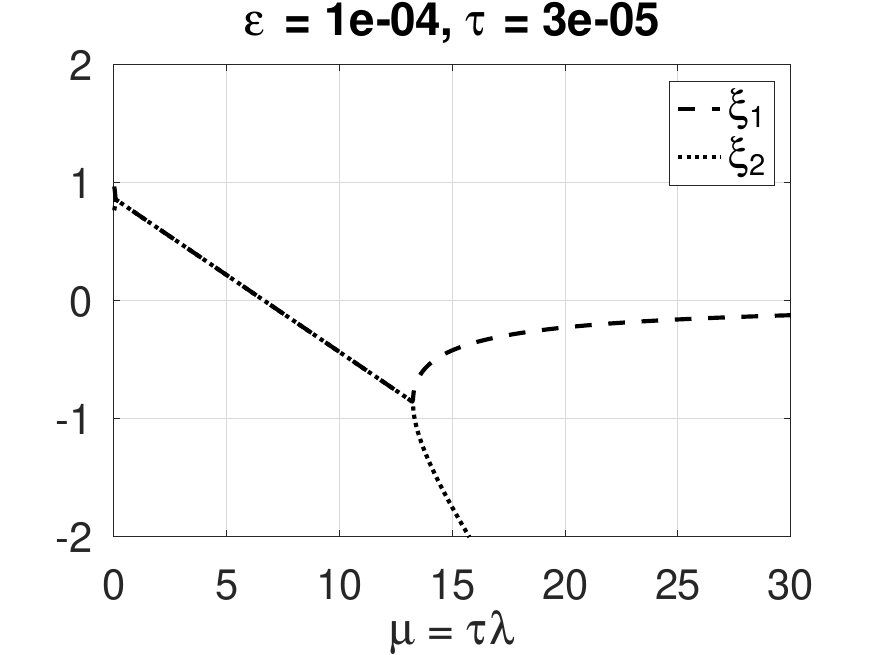}%
\includegraphics[width=0.48\linewidth]{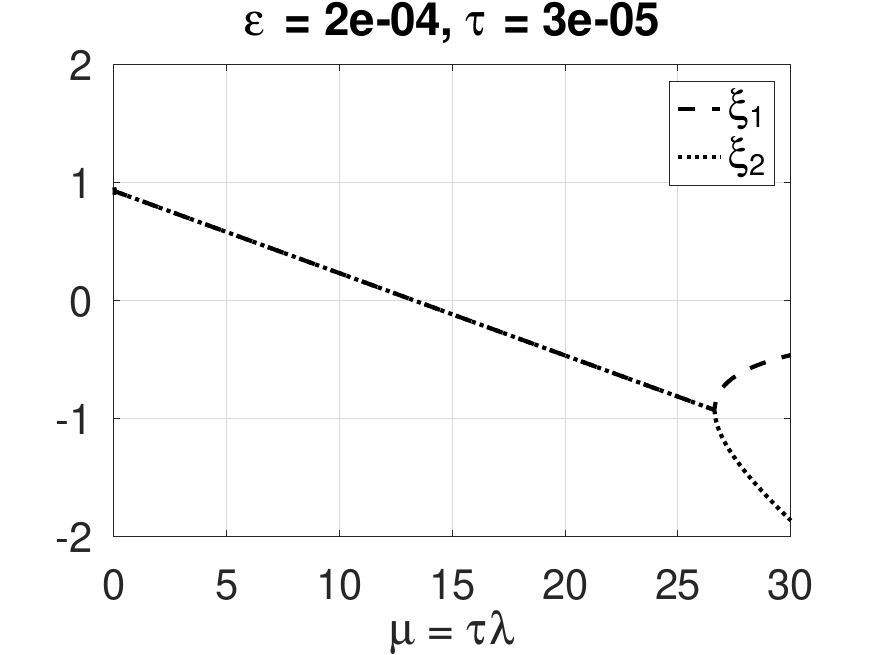}
\caption{Top plots: the first row entries $S_{11}$, $S_{12}$ of the 
amplification matrix~$S$ versus a range of~$\mu=\tau\lambda$
for $\tau=3\cdot10^{-5}$, $\vareps=2\cdot10^{-4}$ (left plot)
and $\vareps=1\cdot10^{-4}$ (right plot).
Bottom plots: the real parts of the eigenvalues~$\xi_{1,2}$ of the 
amplification matrix versus the same $\mu$ range,
for the same $\vareps$ and $\tau$ values as at the top plots.
The curves for $\xi_1$ and $\xi_2$ coincide if $\xi_{1,2}\in\Cc$.}
\label{f:Sij}
\end{figure}

\begin{remark}
\label{r:monot}
The matrix~$S=S_j$ can, in fact, be seen as a stability function $S=S(\mu)$
of the HM scheme and, hence, it is instructive to have a closer look
at its entries.  The two top plots in Figure~\ref{f:Sij} show 
the first row entries
of the matrix~$S$ versus a range of $\mu$.  
As a reference, at these plots we also show the exact solution
operator $e^{-\mu}$ and the stability function $1/(\mu+1)$ 
of the implicit Euler scheme. 
Inspecting these plots, we see that the HM is not monotone, i.e.,
it does not guarantee that its solution remains nonnegative
for initial vectors $y^0$, $y^1$ with nonnegative entries.
Indeed, since $S_{12}$ does not depend on~$\mu$ and can be negative,
the values $y^0>0$ and $y^1>0$ can be chosen such that 
$y^2 = S_{11}y^1 + S_{12}y^0<0$ even for arbitrarily small $\tau>0$.  
\end{remark}

At the two bottom plots in Figure~\ref{f:Sij}, 
the eigenvalues~\eqref{xi12} of the amplification matrix $S$ 
are shown, for the same parameters as in the corresponding top plots.
As we see, for some $\mu$ the entry $S_{11}$ can larger than~$1$ in absolute 
value, even though the eigenvalues $|\xi_{1,2}|<1$.  Hence, $\|S\|$
can exceed~1 in this case. 
 
Indeed, 
since the amplification matrix~$S$ is not symmetric, its eigenvalues alone 
do not determine the complete stability picture of the scheme. 
As shown in~\cite{CoremDitkowski2012}, there are stability
problems in the HM scheme when the eigenvalues $\xi_{1,2}$ of the 
matrices~$S_j$ are approximately equal. It is easy to check that
for $\xi_1=\xi_2$ the matrix~$S_j$ is not diagonalizable and,
as will be seen in numerical tests below, 
for $\xi_1\approx\xi_2$ the powers $S_j^n$ (and, hence, $S^n$) 
exhibit a significant growth in norm with the time step number~$n$.  
It is not difficult to check that the growth
of $\|S^n\|$ is proportional to $|\xi_1-\xi_2|^{-1}$.
From~\eqref{xi12} we see that
$$
|\xi_1 - \xi_2| = 
\left|\frac{2\sqrt{ (2\epst - \mu)^2 + 1 - 4\epst^2}}{1+ 2\epst}\right| =
\frac{2\left|\sqrt{ \mu^2 + 1 - 4\epst\mu}\right|}{1+ 2\epst},
$$
so that, for sufficiently small $\mu=\tau\lambda$,
\begin{equation}
\label{diff_est}
|\xi_1 - \xi_2|^{-1}\approx
\frac{1+ 2\epst}{2\left|\sqrt{1 - 4\epst\mu}\right|}
=\frac{\tau+ 2\vareps}{2\tau\left|\sqrt{1 - 4\vareps\lambda}\right|}.
\end{equation}
In practice, this estimate usually turns out to be a good approximation
for a whole range of realistic values of $\mu$.  

\section{Numerical tests}
\label{s:num}
Assuming that the matrix ODE in~\eqref{IVP} stems from a space discretization
of a parabolic PDE, we consider a scalar test IVP which is 
obtained from~\eqref{IVP} by applying the diagonalizing 
transformation~$Q$, see~\eqref{diag}, i.e., the test problem is
\begin{equation}
\label{IVPt}
y'(t) = - A y(t) + f(t), \qquad y(0) =y^0,
\qquad 0\leqs t\leqs T,  
\end{equation}
with $A>0$ being a positive scalar, $f(t)\equiv 0$, and the exact solution 
$y(t)=e^{-tA}$, $y^0=1$.  The corresponding HM approximation~\eqref{IVP_HM} 
has then the exact solution
$$
\yt(t) = C_1 e^{\tilde{\lambda}_1t} + C_2 e^{\tilde{\lambda}_2t} ,
$$
where
\begin{gather*}
\tilde{\lambda}_{1,2} = 
\frac{-1 \pm \sqrt{D}}{2\vareps}, \quad 
D = 1-4A\vareps,
\quad
C_2 = \frac{\tilde{\lambda}_1 + A}{\tilde{\lambda}_1 - \tilde{\lambda}_2},
\quad
C_1 = 1 - C_2,
\end{gather*}
and the parameters are chosen to have $D>0$.  
The initial conditions in~\eqref{IVP_HM} are set according to
the exact solution~$\yt(t)$.
We set $A=10^3$, $\vareps=2\cdot 10^{-4}$, and $T=3\cdot 10^{-3}$. 

First, 
in Figure~\ref{f:HMerr} we plot the error of the HM approximation~\eqref{et},
its upper estimate~\eqref{et_est} and the value $\vareps\yt''(t)$,
used in the estimate, versus time.  The error estimate is computed for 
$\omega=A$ and is, therefore, rather sharp.  This will not be the case
in more realistic settings.  As we see, $\vareps\yt''(t)$ can reach relatively
large values even for this moderate value of $\lambda_{\max}=10^3$
and, in fact, we see in the tests that its maximum is proportional to
$\lambda_{\max}^2$. 

\begin{figure}
\includegraphics[width=0.48\linewidth]{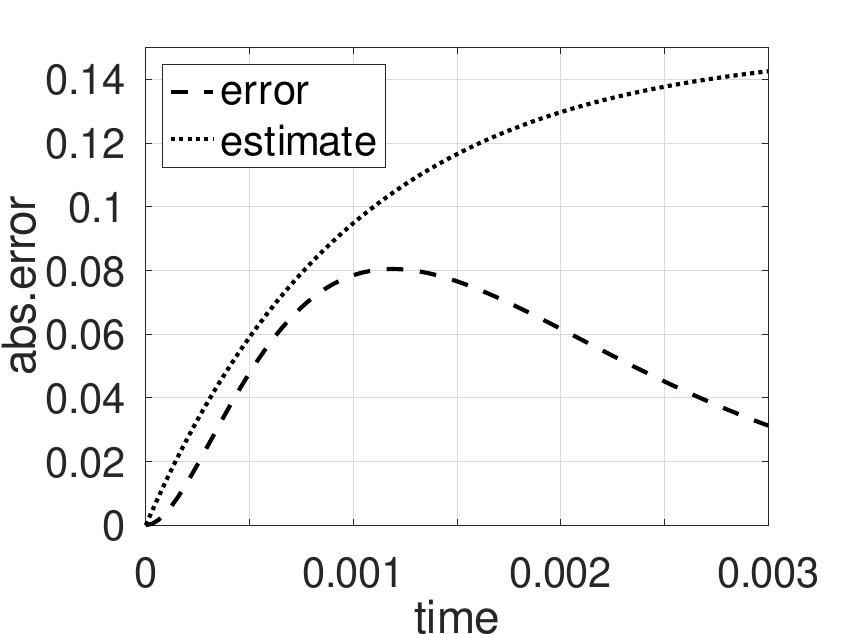}%
\includegraphics[width=0.48\linewidth]{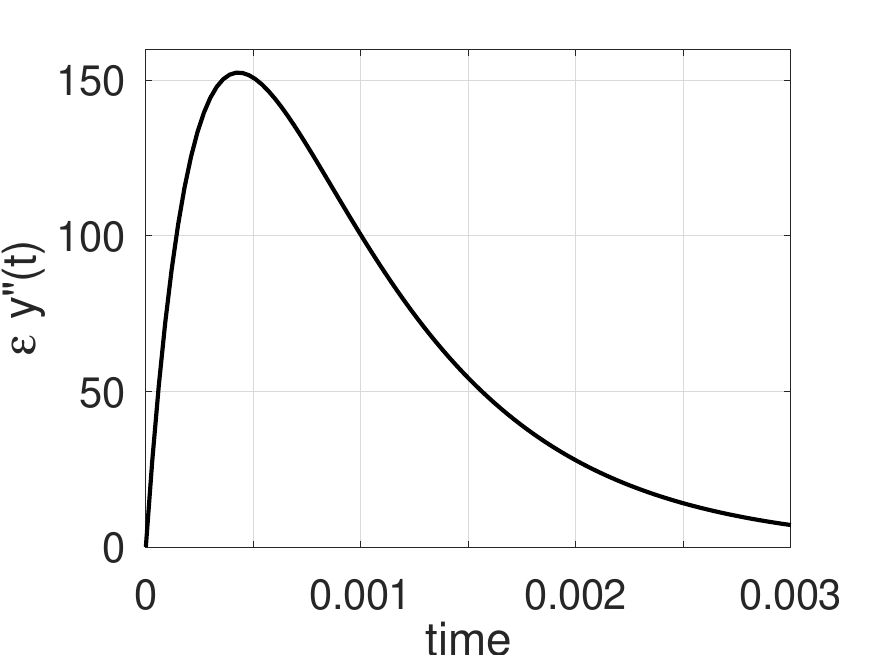}
\caption{Left plot:
the HM error $\et(t)$ and its upper estimate~\eqref{et_est}
versus time.
Right plot: $\vareps\yt''(t)$ versus time.
The plots are obtained for $\vareps=2\cdot10^{-4}$, 
$\omega=\lambda_{\max}=1000$.}
\label{f:HMerr}
\end{figure}

In Figure~\ref{f:loc_glob}, we plot the local and global
(i.e., computed at final time $T=10^{-3}$) numerical errors
of the HM scheme computed for the time step sizes
$\tau=3\cdot10^{-5}$, $\tau/2$, $\tau/4$, and $\tau/8$.
The HM scheme error is computed as the relative error with respect 
to the exact solution~$\yt(t)$ of the HM approximation, 
i.e., the reported error values are $|\yt(t_n) - y^n|/|\yt(t_n)|$.
While the local error exhibits the expected fourth order convergence,
the error at the final time does not converge with order~3, as expected.
It turns out that this convergence order reduction is caused by the growth
of the norms $\|S^n\|$, as $\tau$ decreases, of the HM amplification matrix~$S$.
The actual values $\|S^n\|$ versus time moments $t_n=n\tau$ are shown 
at the left plot in Figure~\ref{f:Sn}.
At the right plot in Figure~\ref{f:Sn}, the values 
$\max_{1\leqs n\leqs T/\tau}\|S^n\|$ are plot versus the 
estimate~\eqref{diff_est} for the reciprocal of the eigenvalue
distance~$|\xi_1-\xi_2|^{-1}$. 
As we see, the estimate~\eqref{diff_est} can serve as a good
indicator for the growing amplification matrix powers. 

\begin{figure}
\includegraphics[width=0.48\linewidth]{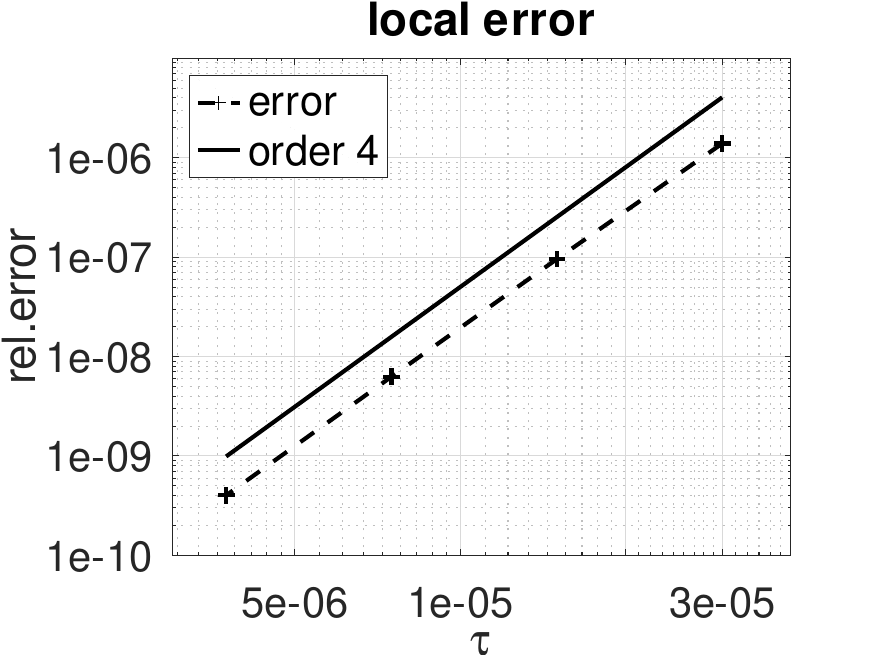}%
\includegraphics[width=0.48\linewidth]{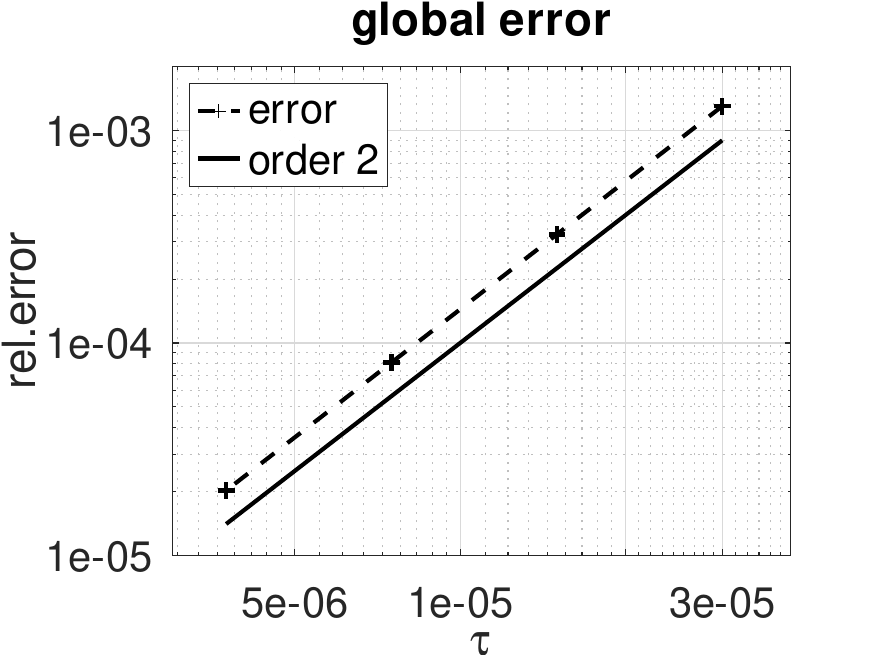}
\caption{Local (left plot) and global (right plot) 
error values of the HM scheme versus the time step size $\tau$.
The plots are obtained for $\vareps=2\cdot10^{-4}$, 
$\lambda_{\max}=1000$, $T=3\cdot10^{-3}$.}
\label{f:loc_glob}
\end{figure}

\begin{figure}
\includegraphics[width=0.48\linewidth]{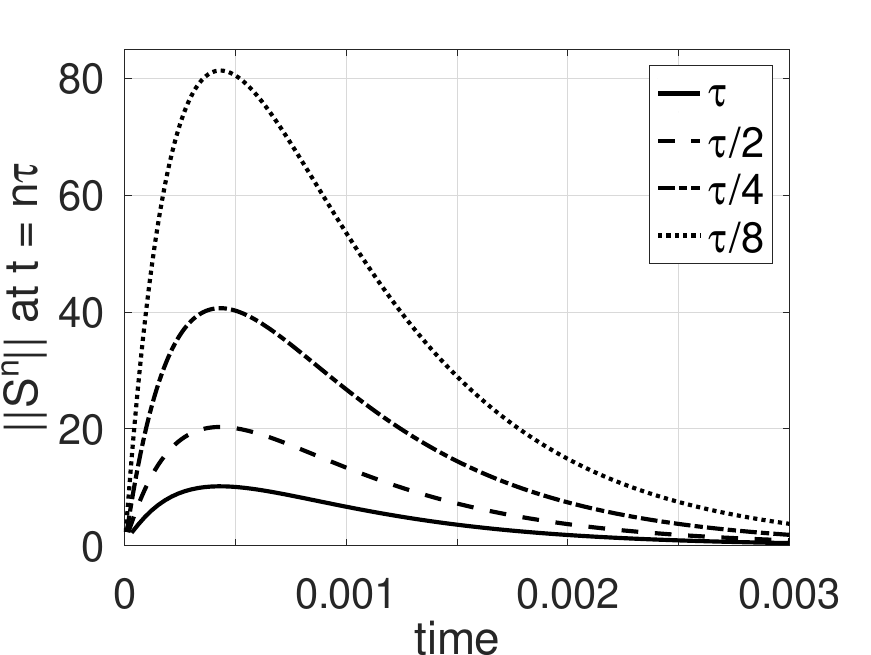}%
\includegraphics[width=0.48\linewidth]{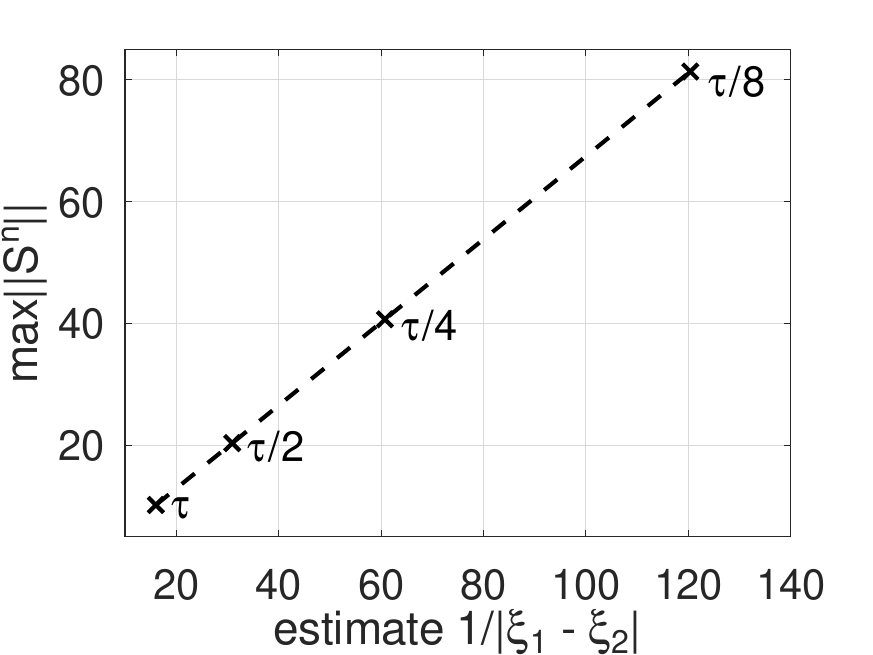}
\caption{Left plot: 
norms $\|S^n\|$ for the amplification matrix $S$ versus time.
Right plot: $\max_{1\leqs n\leqs T/\tau}\|S^n\|$, 
with $n$ being the time step number,
versus the estimate~\eqref{diff_est} for $|\xi_1-\xi_2|^{-1}$.
The plots are obtained for $\vareps=2\cdot10^{-4}$, 
$\lambda_{\max}=1000$, $\tau=3\cdot10^{-5}$.}
\label{f:Sn}
\end{figure}

We now show that our the numerical results presented here for
the scalar test problem~\eqref{IVPt} remain valid for space discretized PDEs.   
In the left plot of Figure~\ref{f:heat1D}, the values $\|S^n\|$
for the amplification matrix $S$ are given for
1D heat equation $u_t = u_{xx}$ on the space interval $x\in(-\pi,\pi)$
with homogeneous Dirichlet boundary conditions, discretized in space
by standard second order finite differences on uniform grid with 
$n_x=100$ points.  Hence, the space grid step size is $h=2\pi/n_x$ 
and we have 
$\|A\|_2=\lambda_{\max}\approx 4h_x^{-2}=n_x^2/\pi^2\approx 1013$.  
In the test runs we use the values $\vareps=2\cdot10^{-4}$,  
$\tau=3\cdot10^{-5}$ for the time step size, and 
$T=3\cdot10^{-3}$ for the final time.
At the right plot of the Figure,
global time errors for the same space discretized heat equation are shown,
where the initial value vector~$y^0$ is formed by the grid values
of the function $\sin(x) + \sin(2x) + \sin(3x)$. 
We observe the same effects as for the scalar test problem
(see the left plot in Figure~\ref{f:Sn} and the right plot
in Figure~\ref{f:loc_glob}): 
$\|S^n\|$ increase with decreasing $\tau$, while the 
global time error
tends to converge with order~2.

\begin{figure}
\includegraphics[width=0.48\linewidth]{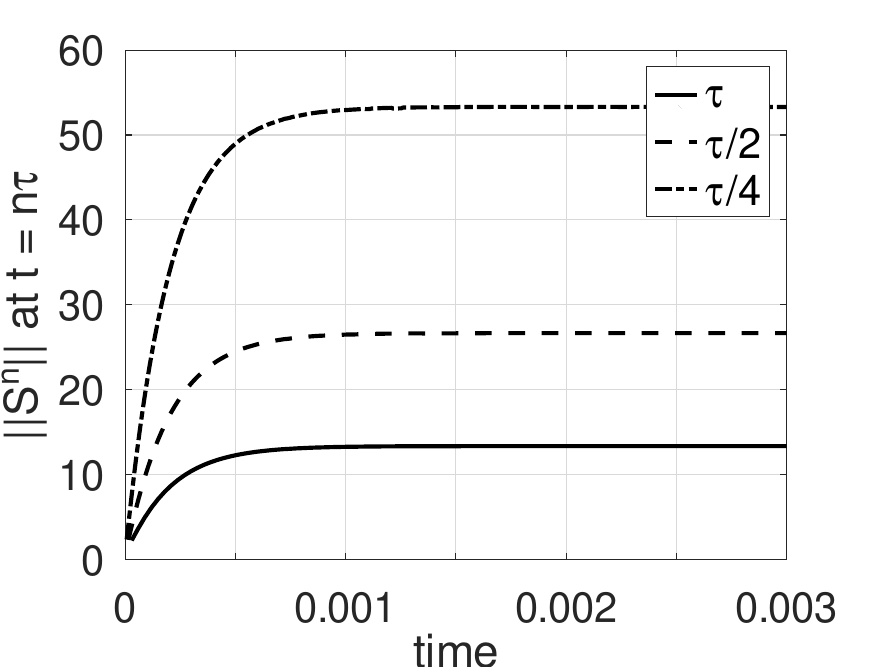}%
\includegraphics[width=0.48\linewidth]{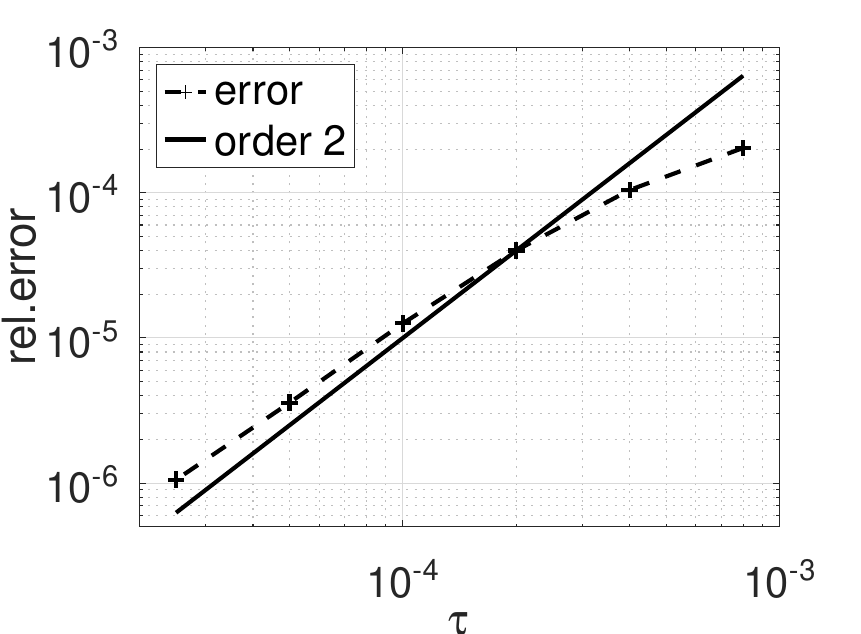}
\caption{Left plot: 
norms $\|S^n\|$ for the amplification matrix $S$ versus time
for the 1D heat equation.
Right plot: 
error values of the HM scheme versus the time step size $\tau$
for the 1D heat equation.
The plots are obtained for $\vareps=2\cdot10^{-4}$, 
$\lambda_{\max}=1033$, $\tau=3\cdot10^{-5}$, $T=3\cdot10^{-3}$.}
\label{f:heat1D}
\end{figure}

From the estimate~\eqref{diff_est}, serving as the 
indicator for $\|S^n\|$, we see that
\begin{equation*}
\label{dt_est}
|\xi_1 - \xi_2|^{-1} = O(\tau^{-1}),
\end{equation*}
which is in accordance with the fact that $\max_{1\leqs n\leqs T/\tau}\|S^n\|$ 
grows with $\tau\rightarrow 0$.  We also see that
by setting $\vareps=\tilde{C}\tau$, for a constant $\tilde{C}>0$,
we can remove the $O(\tau^{-1})$ dependence in~\eqref{diff_est}.
However, by further consideration, this turns out to be not such a 
good idea as it appears at first sight.  There are two reasons
for this.  First, as seen in the proof of Proposition~\ref{p:stab}, 
in real life problems the HM amplification matrix~$S$
is similar to a block-diagonal matrix with $2\times 2$ diagonal
blocks~$S_j$.  These blocks correspond to a set of $\mu_j=\tau\lambda_j$,
formed by all the eigenvalues $\lambda_j\in[0,\lambda_{\max}]$ of~$A$.  
Clearly, $S^n$ is then similar to  a block-diagonal matrix with
diagonal blocks~$S_j^n$.
The dependence of the blocks $S_j$ and its powers on~$\mu$
appears to be rather nontrivial, see Figure~\ref{f:S_mu}.
The plots in Figure~\ref{f:S_mu} show dependence of the norms
$\|S_j^p\|$ on $\mu$ for various~$p$, with the blocks~$S_j$ defined 
by formula~\eqref{Sj}.  Note that the choice $\vareps=\tilde{C}\tau$ 
means that $\epst=\tilde{C}$ is kept constant with decreasing $\tau$.
As the plots in Figure~\ref{f:S_mu} show, the norms
$\|S_j^p\|$ can exhibit a significant growth as $\mu=\tau\lambda$
gets smaller either due to decreasing $\tau$ or for a smaller $\lambda$.
Hence,  
mitigating a growth $S_j^n$ for a certain $\mu_j$ 
by the choice $\vareps=\tilde{C}\tau$ 
can make
a growth $S_i^n$ for some other $\mu_i$ more severe.
This is, apparently, why the powers $S^n$ 
in Figure~\ref{f:heat1D} do not decrease in norm as powers 
of a single block eventually do in Figure~\ref{f:Sn}---the former 
reflect the worst 
case behavior of the blocks~$S_j^n$ for the whole bunch of~$\mu_j$.  

\begin{figure}
\includegraphics[width=0.48\linewidth]{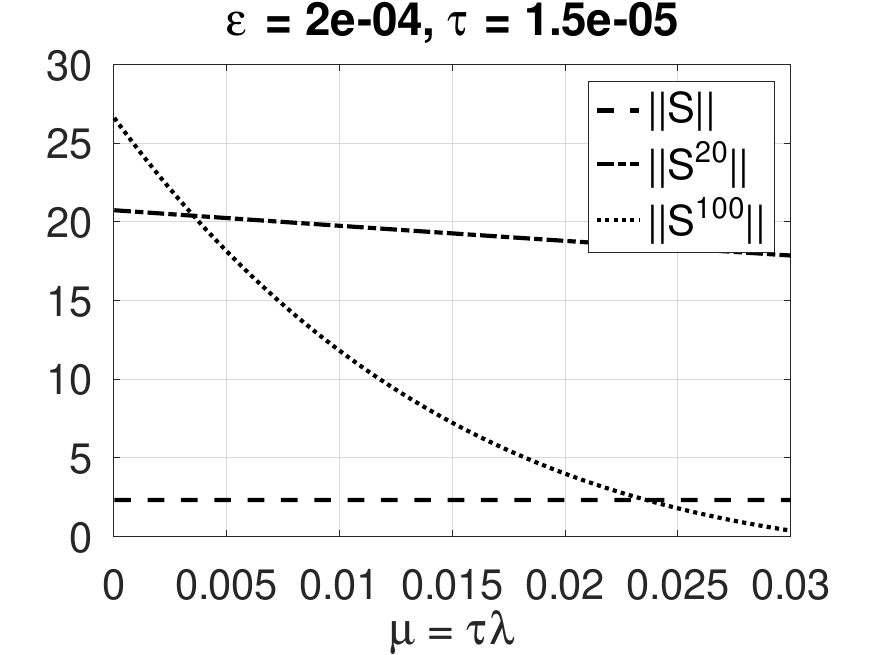}%
\includegraphics[width=0.48\linewidth]{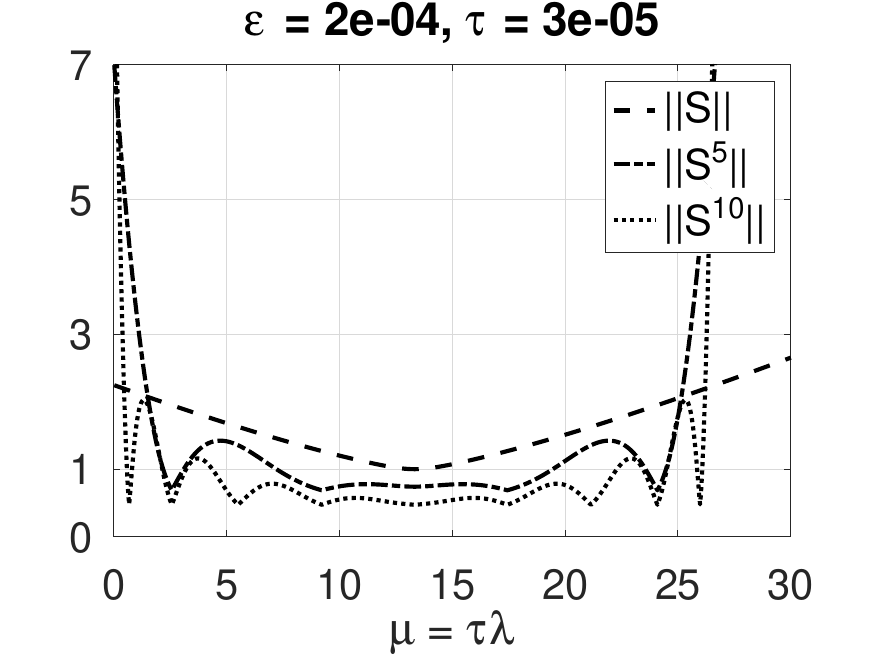}
\caption{A sample dependence of the powers of the $2\times 2$ blocks $S^n$ in
norm on~$\mu=\tau\lambda$.  
Left plot: $\vareps=2\cdot10^{-4}$, $\tau=3\cdot10^{-5}$,
right plot: $\vareps=2\cdot10^{-4}$, $\tau=1.5\cdot10^{-5}$.
Note different ranges of $\mu$ in the two plots.
The stability condition~\eqref{stab2} holds
for $\mu<\mu_{\max}=4\vareps/\tau$ with 
$\mu_{\max}=160/3$ (left plot) and $\mu_{\max}=80/3$ (right plot). 
The right plot corresponds to the $\tau/2=1.5\cdot10^{-5}$ curve
in the left plot of Figure~\ref{f:Sn}, made for $\mu=\lambda\tau=0.015$.}
\label{f:S_mu}
\end{figure}

The other reason why taking $\vareps=\tilde{C}\tau$ may not be a good idea
is that this choice effectively forces the time step size in the HM scheme
be restricted as (cf.~stability condition~\eqref{stab2})
\begin{equation}
\label{expl}
\tau < 4\frac{\vareps/\tau}{\lambda_{\max}} = \frac{4\tilde{C}}{\lambda_{\max}}.
\end{equation}
This restriction is similar to the one in the conventional explicit 
schemes, such as the explicit Euler method.  In particular,
for parabolic problems, where $\lambda_{\max}\approx 4 h^{-2}$, the time
step size in the HM scheme is thus to be restricted as $\tau<\tilde{C}h^2$.
 
Finally, we briefly comment on the 
suggestion from~\cite{ChetverushkinGulin2012},
followed 
in~\cite{ChetverushkinOlkhovskaya2020,CheOlkGas23,ChetverushkinOlkhovskayaGasilov23a}, 
to choose $\vareps=Kh$ with $K$ related to the reciprocal of the 
characteristic process rate (recall that for the heat equation $K$ is 
the heat conduction coefficient, with $\|A\|\approx 4K/h^2$,
cf.~\eqref{stab1a},\eqref{restr1}).
As follows from the accuracy discussions above,
to reach a higher accuracy with $h$ decreasing, typically one has to
decrease~$\tau$ accordingly.  Hence, one should choose, at least, $\tau=O(h)$ 
which, in combination with $\vareps=O(h)$, 
effectively yields the case $\vareps=O(\tau)$ discussed above.
Another possible problem with setting $\vareps=Kh$ is that
$K$ can be quite large, especially in stiff problems,
thus preventing $\vareps$ from being sufficiently small
to keep the HM error bounded.
This may potentially restrict the application area of the HM scheme
to non-stiff or moderately stiff problems.
The last observation seems to be in accordance with the comment made 
in~\cite[Chapter~IV.3.1]{HundsdorferVerwer:book}
for the original Du Fort--Frankel scheme that it can be useful 
mainly for solving problems with small diffusion coefficients
to a low accuracy.

\section{Conclusions}
\label{s:concl}
In this note we have reevaluated accuracy and stability properties
of the HM (hyperbolic model) scheme~\eqref{HM}.  
Formulated in a specific way (see discussion in the Introduction), 
the HM scheme has also been known as the generalized Du Fort--Frankel
scheme~\cite{GottliebGustafsson1976,CoremDitkowski2012}.
The following conclusions can be made.

\begin{enumerate}
\item 
The HM scheme~\eqref{HM} was considered
in the~1971 book~\cite[Chapter~VI, \S~2]{Samarskii_book1971},
where the scheme is given for $\vareps=\kappa\tau^2$.
As noted there by Samarskii, 
the Du~Fort and Frankel scheme~\cite{DuFortFrankel1953}
can be seen as a particular case of his scheme~\eqref{HM1}.

\item 
The HM scheme can be seen as a numerical method for solving 
the hyperbolically perturbed problem~\eqref{IVP_HM}, usually
called the HM approximation.
Its error with respect to the original problem~\eqref{IVP}, 
is proportional to $\vareps\max_{s\in[0,t]}\|\yt''(s)\|$ 
(see, e.g.,~\cite{RepinChetverushkin2013} or
relation~\eqref{et_est} derived here).  
As seen in numerical tests, the value of
$\max_{s\in[0,t]}\|\yt''(s)\|$ is typically proportional to $\|A\|^2$
(in parabolic problems 
$\sim 1/h^4$, with $h$ being the space grid step size). 

\item 
For symmetric positive definite $A$, the eigenvalues of the
amplification matrix in the HM scheme are smaller than one
in absolute value if and only if condition~\eqref{stab2} holds.
This condition coincides with Samarskii stability 
condition~\cite[Chapter~VI, \S~2]{Samarskii_book1971}.

\item 
The HM scheme is not monotone in the sense that it does not preserve 
nonnegativity solution properties of the original problem~\eqref{IVP}
even for arbitrarily small $\tau>0$ (see Remark~\ref{r:monot}).
 
\item 
The amplification matrix norm $\|S\|_2$ of the HM scheme can exceed one,
even if the Samarskii stability condition holds.
The powers $S^n$ can exhibit a significant 
growth in norm with the time step number~$n$,
and, as has been typically observed,
$\max_n\|S^n\|=O(\tau^{-1})$ (see Figure~\ref{f:Sn}).
This can corrupt convergence properties of the scheme.  
For a fixed~$n$, $\|S^n\|$
is not monotone with respect to $\mu=\tau\lambda$
(with $\lambda$ being an eigenvalue of~$A$).
For a fixed~$\mu$, $\|S^n\|$
is not monotone with respect to $n$ (see Figure~\ref{f:S_mu}).

\item
The HM scheme amplification matrix $S$ can be shown to be similar
to a block diagonal matrix with $2\times 2$ diagonal blocks $S_j$,
see relation~\eqref{Sj}.  The growth of $S^n$ in norm is observed
when some of the blocks $S_j$ become (close to) nondiagonalizable
(or, equivalently, when the eigenvalues~$\xi_{1,2}$ in the blocks~$S_j$
are mutually close)~\cite{CoremDitkowski2012}. 
To monitor the eventual growth $S^n$ in practice, 
we have proposed an estimate~\eqref{diff_est}, which can be evaluated
for a range $\lambda\in[\lambda_{\min},\lambda_{\max}]$
of the eigenvalues of~$A$.

\item
Adjusting $\vareps$ to mitigate the growth $S^n$, in particular,
taking it proportional to the time step size~$\tau$, does not
seem to yield a clear improvement (as it can simultaneously 
worsen the situation for some other diagonal blocks).  
The choice $\vareps=\tilde{C}\tau$ transforms
the Samarskii stability condition~\eqref{stab1} into 
an asymptotically tight restriction $\tau < \tilde{C}h^2$, see~\eqref{expl}.
Similar problems can be potentially expected for the choice $\vareps=K h$.

\end{enumerate}

A relevant question for a future research is whether an adaptive
choice of the small parameter $\vareps$ 
can help to resolve these problems.
For instance, it would be interesting to consider
a scheme~\eqref{HM} with $\vareps$ replaced by 
a diagonal matrix.

\section*{Acknowledgments}
The author thanks Victor Timofeevich Zhukov for proposing
this research topic and for many useful discussions.

The use of the hybrid supercomputer K-100 facilities installed 
in the Supercomputer center of collective usage of KIAM RAS
is acknowledged.

\section*{Funding}
This work is supported by ongoing research funding within the State assignment 
framework of the Keldysh Institute of Applied Mathematics of Russian Academy of 
Sciences.

\section*{Conflict of interest}
The author of this work declares that he has no conflicts of
interest.

\bibliography{hm,my_bib}

\def\ocirc#1{\ifmmode\setbox0=\hbox{$#1$}\dimen0=\ht0 \advance\dimen0
  by1pt\rlap{\hbox to\wd0{\hss\raise\dimen0
  \hbox{\hskip.2em$\scriptscriptstyle\circ$}\hss}}#1\else {\accent"17 #1}\fi}
  \def\cprime{$'$}
\begin{thebibliography}{10}

\bibitem{Samarskii_book1971}
A.~A. Samarskii, {\em Introduction to the theory of difference schemes}.
\newblock Moscow: Nauka, 1971.
\newblock In Russian.

\bibitem{ZlotnikChetverushkin2008}
A.~A. Zlotnik and B.~N. Chetverushkin, ``Parabolicity of the quasi-gasdynamic
  system of equations, its hyperbolic second-order modification, and the
  stability of small perturbations for them,'' {\em Comput.\ Math.\ and Math.\
  Phys.}, vol.~48, pp.~420--446, 2008.
\newblock \url{https://doi.org/10.1134/S0965542508030081}.

\bibitem{CheOlkGas23}
B.~N. Chetverushkin, O.~G. Olkhovskaya, and V.~A. Gasilov, ``An explicit
  difference scheme for non-linear heat conduction equation,'' {\em Math.\
  Models Comput.\ Simul.}, vol.~15, no.~3, pp.~529--538, 2023.
\newblock \url{https://doi.org/10.1134/S2070048223030031}.

\bibitem{ChetverushkinOlkhovskayaGasilov23a}
B.~N. Chetverushkin, O.~G. Olkhovskaya, and V.~A. Gasilov, ``Three-level scheme
  for solving the radiation diffusion equation,'' {\em Dokl.\ Math.}, vol.~108,
  no.~1, pp.~320--325, 2023.
\newblock \url{https://doi.org/10.1134/S1064562423700837}.

\bibitem{DuFortFrankel1953}
E.~C. {Du Fort} and S.~P. Frankel, ``Stability conditions in the numerical
  treatment of parabolic differential equations,'' {\em Mathematical Tables and
  Other Aids to Computation}, vol.~7, no.~43, pp.~135--152, 1953.
\newblock \url{https://doi.org/10.2307/2002754}.

\bibitem{SamarskiiGulin_book1973}
A.~A. Samarskii and A.~V. Gulin, {\em Stability of difference schemes}.
\newblock Moscow: Nauka Publisher, 1973.
\newblock Available at \url{https://samarskii.ru}. In Russian.

\bibitem{RichtMorton}
R.~D. Richtmyer and K.~W. Morton, {\em Difference methods for initial-value
  problems}.
\newblock Interscience Publishers John Wiley \& Sons, Inc., New
  York-London-Sydney, 1967.
\newblock Second edition.

\bibitem{GottliebGustafsson1976}
D.~Gottlieb and B.~Gustafsson, ``Generalized {Du Fort}--{Frankel} methods for
  parabolic initial-boundary value problems,'' {\em SIAM J. Numer.\ Anal.},
  vol.~13, no.~1, pp.~129--144, 1976.
\newblock \url{https://doi.org/10.1137/0713015}.

\bibitem{HundsdorferVerwer:book}
W.~Hundsdorfer and J.~G. Verwer, {\em Numerical Solution of Time-Dependent
  Advection-Diffusion-Reaction Equations}.
\newblock Springer Verlag, 2003.

\bibitem{CoremDitkowski2012}
N.~Corem and A.~Ditkowski, ``New analysis of the {Du Fort}--{Frankel}
  methods,'' {\em J. Sci.\ Comput.}, vol.~53, pp.~35--54, 2012.
\newblock \url{https://doi.org/10.1007/s10915-012-9627-2}.

\bibitem{IlyinRykov2016}
A.~A. Ilyin and Y.~G. Rykov, ``On the closeness of trajectories for model
  quasi-gasdynamic equations,'' {\em Dokl.\ Math.}, vol.~94, no.~2,
  pp.~543--546, 2016.
\newblock \url{https://doi.org/10.1134/S1064562416040256}.

\bibitem{Moiseev_ea2018}
T.~E. Moiseev, E.~E. Myshetskaya, and V.~F. Tishkin, ``On the closeness of
  solutions of unperturbed and hyperbolized heat equations with discontinuous
  initial data,'' {\em Dokl.\ Math.}, vol.~98, no.~1, pp.~391--395, 2018.
\newblock \url{https://doi.org/10.1134/S1064562418050277}.

\bibitem{Chetverushkin2004book}
B.~N. Chetverushkin, {\em Kinetic Schemes and Quasi-Gasdynamic System of
  Equations}.
\newblock Moscow: MAKS Publisher, 2004.
\newblock In Russian.

\bibitem{ChetverushkinGulin2012}
B.~N. Chetverushkin and A.~V. Gulin, ``Explicit schemes and numerical
  simulation using ultrahigh-performance computer systems,'' {\em Dokl.\
  Math.}, vol.~86, no.~2, pp.~681--683, 2012.
\newblock \url{https://doi.org/10.1134/S1064562412050213}.

\bibitem{ChetverushkinOlkhovskaya2020}
B.~N. Chetverushkin and O.~G. Olkhovskaya, ``Modeling of radiative heat
  conduction on high-performance computing systems,'' {\em Dokl.\ Math.},
  pp.~172--175, 2020.
\newblock \url{https://doi.org/10.1134/S1064562420020088}.

\bibitem{HochbruckOstermann2010}
M.~Hochbruck and A.~Ostermann, ``Exponential integrators,'' {\em Acta Numer.},
  vol.~19, pp.~209--286, 2010.

\bibitem{ODE1_HNW}
E.~Hairer, S.~P. N{\o}rsett, and G.~Wanner, {\em Solving Ordinary Differential
  Equations I. {Nonstiff} Problems}.
\newblock Springer Series in Computational Mathematics 8, Springer--Verlag,
  1987.

\bibitem{RepinChetverushkin2013}
S.~I. Repin and B.~N. Chetverushkin, ``Estimates of the difference between
  approximate solutions of the {Cauchy} problems for the parabolic diffusion
  equation and a hyperbolic equation with a small parameter,'' {\em Dokl.\
  Math.}, vol.~88, no.~1, pp.~417--420, 2013.
\newblock \url{https://doi.org/10.1134/S1064562413040157}.

\bibitem{Samarskii_book}
A.~A. Samarskii, {\em The theory of difference schemes}.
\newblock CRC Press, 2001.

\bibitem{GodunRyab}
S.~K. Godunov and V.~S. Ryabenkii, {\em Difference Schemes. An Introduction to
  the Underlying Theory}.
\newblock Elsevier Science, 1987.

\end{thebibliography}
\bibliographystyle{ieeetr}

\end{document}